\documentclass[11pt]{amsart}

\usepackage{fullpage}

\usepackage{amsmath}
\usepackage{amssymb}
\usepackage{amsthm}
\usepackage{amscd}

\usepackage{hyperref}

\newcommand{\SO}{\mathrm{SO}}
\newcommand{\SL}{\mathrm{SL}}
\newcommand{\vol}{\operatorname{vol}}

\title{A note on maximal lattice growth in $\SO(1,n)$}

\author{Jean Raimbault}

\address{Institut de Math\'ematiques de Jussieu \\Unit\'e Mixte de Recherche 7586 du CNRS \\Universit\'e Pierre et Marie Curie \\4, place Jussieu 75252 Paris Cedex 05, France.}\email{raimbault@math.jussieu.fr}

\begin{document}

\numberwithin{equation}{section}

\newtheorem{theo}{Theorem}[section]
\newtheorem{lem}[theo]{Lemma}
\newtheorem{prop}[theo]{Proposition}
\newtheorem{rmk}[subsection]{Remark}
\newtheorem{cor}[theo]{Corollary}

\begin{abstract}
We show that the number of maximal lattices with volume less than $v$ is 
at least exponential in $v$. To do this we construct families of 
noncommensurable Gromov/Piatetski-Shapiro manifolds having the same volume.
\end{abstract}

\maketitle

\section{Introduction}

In their famous paper \cite{GPS} Gromov and Piatetski-Shapiro have shown 
how to construct examples of nonarithmetic hyperbolic $n$-manifolds in any 
dimension $n\ge 3$ by gluing together 
submanifolds of noncommensurable standard arithmetic manifolds. Here we aim 
to show that their methods 
have other applications, for example we construct arbitrarily 
large families of noncommensurable manifolds with the same volume:

\begin{theo}
Let $n\ge 3$; there exists constants $V,K>0$ depending on $n$ such that for all 
integers $m\ge V$ there is a family of at least $2^m$ noncommensurable 
(nonarithmetic) compact hyperbolic $n$-manifolds all having the same volume 
$v\le Km$.
\label{samevol}
\end{theo}

This can be seen as giving an explicit lower bound in Wang's theorem 
\cite{Wang} that there are only finitely many hyperbolic $n$-manifolds 
($n\ge 4$) with volume less than a 
given $v$, however in that respect it is much less good than existing results. 
For example, it is proven in \cite{BGLM} that for any large integer $m$ there 
are at least $m!$ hyperbolics $n$-manifolds the same volume less than 
$K'm$ (for some $K'$ depending on the dimension $n$). However all these 
manifolds are constructed as coverings of a fixed manifold, so that our result 
is slightly different from theirs or those proved in \cite{Emery} or 
\cite{Zimmermann}, in particular it implies:

\begin{cor}
Let $L_{com}(v)$ be the number of commensurability classes of compact 
hyperbolic $n$-manifolds which have an element with volume less than $v$. 
Then 
\begin{equation*}
\liminf_{v\rightarrow\infty}\frac{\log L_{com}(v)}v > 0.
\end{equation*}
\end{cor}

Since there is at least one maximal subgroup with minimal covolume in each 
commensurability class, 
it also follows from Theorem \ref{samevol} that for $v\ge V$ there 
exists at least $2^{\frac v K}$ maximal (nonarithmetic) uniform lattices 
with covolume less than $v$. Interestingly, the growth of maximal 
\textit{arithmetic} subgroups is subexponential by a result 
of Misha Belopelitsky (cf. \cite{Belopelitsky}), so that the growth 
rate of maximal nonarithmetic subgroups is strictly bigger.

\begin{cor}
Let $L_{\text{max}}(v)$ (resp. $L_{a,\text{max}}(v)$) be the number 
of conjugacy classes of maximal uniform lattices in $\SO(1,n)$ 
(resp. arithmetic ones) with covolume less than $v$. 
Then $L_{\text{max}}$ is at least exponential, in particular we have:
\begin{equation*}
\lim_{v\rightarrow\infty}\frac{\log L_{a,\text{max}}(v)}{\log L_{\text{max}}(v)}=0. \label{compmax}
\end{equation*}
\end{cor}

For $n=3$ the number of lattices with volume less than $v$ is infinite for $v$ 
big enough, but this number becomes finite if we add the condition that the 
systole be bigger than any positive constant. In our theorem the manifolds 
have their systoles bounded below by a positive constant. 
As an aside, we also construct manifolds with arbitrarily large volume the 
size of whose isometry group stays bounded. All our constructions are 
quite costless in technology (we only use the Gromov/Piatetski-Shapiro 
paper). 

\subsection*{Acknowledgments} Part of this work comes from the joint work 
\cite{7S} (see also the announcement in CRAS Tome 349, Fascicule 15-16 or 
\href{http://arxiv.org/abs/1104.5559}{arXiv:1104.5559}) 
where it was used for other purposes.

\section{Gluing manifolds}

\subsection{General construction}
\label{gluing}

Let $r\ge 2$  and $N_1,\ldots,N_r$ be real compact hyperbolic $n$-manifolds 
such that each has totally geodesic boundary, and each boundary is the 
disjoint union of two 
copies of some hyperbolic $(n-1)$-manifold $\Sigma$. We choose 
for each manifold $N_a$ a component $\Sigma_a^+$ of $\partial N_a$, and 
denote the other one by $\Sigma_a^-$; we call $i_{a}^{\pm}$ the corresponding 
embeddings of $\Sigma$ in $\partial N_a$. For an integer $m>0$ we identify 
the cyclic group $\mathbb{Z}/m\mathbb{Z}$ with the set $\{0,\ldots,m-1\}$. 
For $\alpha=(\alpha_0,\ldots,\alpha_{m-1})\in\{1,\ldots,r\}^{\mathbb{Z}/m\mathbb{Z}}$ 
we define $M_{\alpha}$ to be the closed hyperbolic manifold obtained by 
gluing copies of the $N_k$ in the manner prescribed by $\alpha$:
\begin{equation*}
M_{\alpha}=\left(\bigsqcup_{i\in\mathbb{Z}/m\mathbb{Z}}N_{\alpha_i}\times\{i\}\right)/\{\forall i\in\mathbb{Z}/m\mathbb{Z},x\in\Sigma,\:(i_{\alpha_i}^+ x,i)= (i_{\alpha_{i+1}}^- x,i+1)\}.
\end{equation*}
For $k=0,\ldots,m$ we shall take the (abusive) convention of denoting 
by $N_{\alpha_k}$ the submanifold image of $N_{\alpha_k}\times\{k\}$ in 
$N_{\alpha}$ and by $\Sigma_k$ the hypersurface image of 
$i_{\alpha_k}^-(\Sigma)\times\{k\}$. 
In the same way, for a subsequence $(\alpha_k,\ldots,\alpha_{k+l})$ 
we denote by $N_{\alpha_k,\ldots,\alpha_{k+l}}$ the submanifold of 
$N_{\alpha}$ image of the $N_{\alpha_k}\times\{k\},\ldots,
N_{\alpha_{k+l}}\times\{k+l\}$.

\subsection{Interbreeding arithmetic manifolds}

According to the construction of Proposition \ref{millson} below there exists 
a family $M_k, k\ge 1$ of arithmetic hyperbolic closed $n$-manifolds 
such that any two of them are noncommensurable but they all contain an 
embedded nonseparating totally geodesic hypersurface isometric to some 
fixed $\Sigma$. Letting $N_k$ be the completion of 
$M_k-\Sigma$ (so that it has boundary $\Sigma\sqcup\Sigma$), for all 
$r>0$ we can perform with 
$N_1,\ldots,N_r$ the constructions of Section \ref{gluing}. 
(Gromov and Piatetski-Shapiro call this interbreeding of manifolds, and the 
resulting nonarithmetic manifolds are called hybrids).
There is a natural (shift) action of 
$\mathbb{Z}/m\mathbb{Z}$ on $\{1,\ldots,r\}^{\mathbb{Z}/m\mathbb{Z}}$ and it is clear 
that we obtain 
the same manifolds for two collections of indices in the same orbit. 
The main technical result is that this is the only case when they 
are commensurable (recall that 
two Riemannian manifolds are said to be commensurable if they share a finite 
cover up to isometry):

\begin{prop}
Let $\alpha_1,\ldots,\alpha_m,\beta_1,\ldots,\beta_m \in \mathbb{N}$. 
The manifolds $M_{\alpha_1,\ldots,\alpha_m}$ and 
$M_{\beta_1,\ldots,\beta_m}$ are commensurable if and only if there 
is a p such that $\alpha_i=\beta_j$ as soon as $i+p=j\pmod{m}$, i.e. if 
$\alpha$ and $\beta$ are in the same $\mathbb{Z}/m\mathbb{Z}$-shift orbit.
\label{noncommensurable}
\end{prop}

We will prove this in section \ref{proof} below. Now we give the proof of 
our main result.


\subsection{Manifolds with the same volume}

We keep the same notation as above and 
denote $\vol(M_k)=v_k$. Theorem \ref{samevol} follows immediately from the 
case $r=2$ of the next result, which gives a more precise estimate.

\begin{prop}
For all positive integers $r,m$ there exists a family of noncommensurable
closed hyperbolic $n$-manifolds all having the same volume 
$\sum_{i=1}^r mv_i$ which contains exactly $a_m$ elements, and 
$a_m\underset{m\rightarrow\infty}{\sim} m^{-1}(2\pi m)^{-\frac {r-1}2}r^{rm-1}$
\label{volume}
\end{prop}

\begin{proof}
Let $S$ be the set of all manifolds $M_{\alpha_1,\ldots,\alpha_{rm}}$ such 
that for each $k=1,\ldots,r$ exactly $m$ of all $\alpha_i$ are equal to $k$, 
modulo commensurability. All manifolds in this set have volume 
$\sum_{i=1}^r mv_i$. According to Proposition \ref{noncommensurable} the 
cardinality of $S$ is bigger than $(rm)^{-1}$ times the number of such 
collections of $\alpha_i$. This last number is equal to:
\begin{equation*}
\binom{mr}{m}\times\binom{m(r-1)}{m}\times\ldots\times\binom{2m}{m} 
   =\frac{(rm)!}{(m!)^r} \sim_{m\rightarrow\infty} (2\pi m)^{-\frac {r-1}2}r^{rm}
\end{equation*}
where the asymptotic follows from Stirling's formula.
\end{proof}

We get a constant $K$ proportional to the volume $\max(v_1,v_2)$, in particular 
it could be estimated using the explicit construction of Proposition 
\ref{millson} and Prasad's volume formula. Note that taking a bigger $r$ would 
not improve on the 
growth rate since the constant $K\gg\max\vol(v_1,\ldots,v_r)$ would tend to 
infinity too quickly (stricly faster than $\log r$, as follows from 
the estimates in \cite{Belopelitsky}) so that the exponential growth rate 
$r^{1/K}$ would tend to 1.


\subsection{Manifolds with small group of isometries}

We take the same notation as in the preceding subsection, with $r=2$.

\begin{prop}
There exists a family $S_k,k\ge 1$ of closed hyperbolic $n$-manifolds 
such that $\vol(S_k)\rightarrow\infty$ but the groups of isometries 
$\mathrm{Isom}(S_k)$ have uniformly bounded order.
\end{prop}

Note that the systole of the $M_k$ stays constant. I don't know whether there 
exists a sequence of compact hyperbolic $n$-manifolds the systole of which 
tends to infinity and which satisfy the conclusion above.

\begin{proof}
Let $S_k=M_{2,1,\ldots,1}$ with $k$ ones, then 
$\vol(S_k)\ge kv_1$. On the other hand, an isometry of $S_k$ must 
map the $N_2$ piece to itself because of Lemma \ref{nointer}, so that 
$|\mathrm{Isom}(S_k)|\le|\mathrm{Isom}(N_2)|$. It remains to show that the 
group $\mathrm{Isom}(N_2)$ is finite: this is because an isometry is 
determined by its restriction to the boundary, so that 
$|\mathrm{Isom}(N_2)|\le 2|\mathrm{Isom}^+(\Sigma)|<+\infty$.
\end{proof}


\section{Noncommensurability}

\label{proof}

\subsection{Some lemmas}

We recall here some results that will be useful for proving the 
noncommensurability of the manifolds constructed in Section \ref{gluing}.
The main advantage of working with arithmetic manifolds is that the
following criterion for commensurability is available
(see \cite[1.6]{GPS}).

\begin{prop}
If $\Gamma,\Gamma'$ are two arithmetic subgroups in $\SO(1,n)$ such that
the intersection $\Gamma\cap\Gamma'$ is Zariski-dense in $\SO(1,n)$, then
this intersection has finite index in both of them (so that they are in
particular commensurable).
\label{criterion}
\end{prop}

We will also require the following lemma, whose proof is similar to that of 
Corollary 1.7B in \cite{GPS}.

\begin{lem}
Let $n\ge 3$, $W$ a complete hyperbolic $n$-manifold and $U,U'$ two 
submanifolds with compact totally geodesic boundary. Then either the 
intersection $U\cap U'$ has empty interior or it contains an hyperbolic 
$n$-manifold with Zariski-dense fundamental group.
\label{ZD}
\end{lem}

\begin{proof}
We fix a monodromy for the hyperbolic structure on $W$; then for any 
totally geodesic submanifold we get a monodromy as a subgroup of 
$\pi_1(W)\subset\SO(1,n)$. 
Suppose that $U\cap U'$ contains an open set and let $S$ (resp. $S'$) be a 
component of
$\partial U$ (resp. $\partial U'$) which intersects the interior of $U'$ 
(resp. $U$). Then $\pi=\pi_1(S\cap U')$ must be Zariski-dense
in the subgroup $H\cong\SO(1,n-1)$
of $\SO(1,n)$ determined by the compact hyperbolic $n-1$-manifold $S$.
Indeed, let $S_0$ be a connected component of $S\cap U'$. Then the boundary
$\partial S_0$ is a (union of) connected components of $S\cap \partial U'$ and
thus a compact hyperbolic manifold of dimension $n-2$; we choose a component
$S_1$. Lemma 1.7.A of \cite{GPS} tells us that $\pi_1(S_0)$ contains
$\pi_1(S_1)$ as a subgroup of infinite index. Now a discrete subgroup of
$\SO(1,n-1)$ containing a subgroup of infinite index which is a lattice in
some closed subgroup $H_2\cong\SO(1,n-2)$ must be Zariski-dense in
$\SO(1,n-1)$ because $H_2$ is a maximal closed subgroup.
In the same way $\pi'=\pi_1(S'\cap U)$ is Zariski-dense in another
distinct embedding $H'$ of $\SO(1,n-1)$. It follows that
$\langle\pi,\pi'\rangle$ is Zariski-dense in $\SO(1,n)$, and since it is
contained in $\pi_1(U\cap U')$ this proves the lemma.
\end{proof}

As a consequence of these two results we get:

\begin{lem}
Let $N_0,N_1$ be two submanifolds with totally geodesic boundary inside two 
noncommensurable arithmetic hyperbolic $n$-manifolds $M_0, M_1$. 
For any complete hyperbolic manifold $W$, finite covers $N_0',N_1'$ of 
$N_0,N_1$ and isometric embeddings 
$\iota_0,\iota_1:N_0',N_1'\hookrightarrow W$, we have 
that $\iota_0(N_0')\cap\iota_1(N_1')$ has empty interior. 
\label{nointer}
\end{lem}

\begin{proof}
Suppose that some component $N$ of $\iota_0(N_0')\cap\iota_1(N_1')$ contains 
an open set; then by Lemma \ref{ZD}, $\pi_1(N)$ is Zariski-dense in $\SO(1,n)$.
 Thus, the intersection $\pi_1(N_0)\cap\pi_1(N_1)$ contains a Zariski-dense 
subgroup, which contradicts Proposition \ref{criterion} and the fact that 
$M_0,M_1$ are noncommensurable.
\end{proof}


\subsection{Conclusion}

We keep notation as in the statement of Proposition \ref{noncommensurable}. 
For convenience we denote $M=M_{\alpha_1,\ldots,\alpha_m}$ and 
$M'=M_{\beta_1,\ldots,\beta_m}$. Suppose that $M$ is commensurable to $M'$, 
let $W$ be a common finite covering of $M$ and $M'$ with covering maps 
$\pi,\pi'$. We will use the following lemma.

\begin{lem}
If $i\not=j$ are two indices such 
that $\alpha_i\not=\alpha_{i+1}=\ldots=\alpha_j\not=\alpha_{j+1}$ then 
any connected component of $\pi^{-1}(N_{\alpha_i,\ldots,\alpha_j})$ is equal to a 
connected component of $(\pi')^{-1}(N_{\beta_{i'},\ldots,\beta_{j'}})$ 
where $j'-i'=j-i$ and for $l=1,\ldots,j-i$ we have 
$\beta_{i'+l}=\alpha_{i+l}$. 
\label{lemma}
\end{lem}

\begin{proof}
Let $l\in\{0,\ldots,m-1\}$ be any index 
such that $\beta_l\not=\alpha_{i+1}$. Then by Lemma \ref{nointer} we have that 
$(\pi')^{-1}(N_{\beta_l})\cap \pi^{-1}(N_{\alpha_{i+1},\ldots,\alpha_j})$
has empty interior. It follows that for any connected component $N$ of 
$\pi^{-1}(N_{\alpha_{i+1},\ldots,\alpha_j})$ there are 
$\beta_{k+1}=\ldots=\beta_{k+l}=\alpha_{i+1}$ such that $N$ is contained 
in a connected component $N'$ of 
$(\pi')^{-1}(N_{\beta_{k+1},\ldots,\beta_{k+l}})$. But again by Lemma 
\ref{nointer} we get that $N'\cap\pi^{-1}(N_{\alpha_i})$ as well as 
$N'\cap\pi^{-1}(N_{\alpha_{j+1}})$ have empty interior, which forces 
$N'=N$. All that is left to show is that $l=i-j$, which follows from the 
fact that the distance between the boundaries of $N$ and $N'$ is the 
same.
\end{proof}

\begin{proof}[Proof of Proposition \ref{noncommensurable}]
Suppose first that all $\alpha_i$ are equal to some $k$. 
Then $M$ is a cyclic cover of $M_k$, and in particular it is arithmetic. 
It follows that $M'$ is arithmetic too and by Gromov and Piatetski-Shapiro's 
argument if this is the case then all $\beta_i$ must be equal too, 
say to some $k'$. Now we get that $M_k$ and $M_{k'}$ are commensurable and 
it follows that we have $k=k'$.

Now suppose there is an index $0\le i_1\le m-1$ such that 
$\alpha_{i_1}\not=\alpha_{i_1+1}$: we may suppose that $i_1=0$. 
Choose a connected component $N$ of $\pi^{-1}(N_{\alpha_{i_1+1}})$, let 
$S=\pi^{-1}(\Sigma_0)\cap\partial N$. Take a loop $c$ in $M_{\alpha}$ which 
meets each hypersurface $\Sigma_i$ exactly once and transversally and a 
lift $\tilde{c}$ of $c$ to $W$ based at some point on $S$. Then as we go along 
$\tilde{c}$ we meet connected components of 
$\pi^{-1}(N_{\alpha_1,\ldots,\alpha_{i_2}}),\ldots,\pi^{-1}(N_{\alpha_{i_{k-1}+1},\ldots,\alpha_{i_k}})$ 
where $i_k=0$ and we have 
$\alpha_{i_l} \not= \alpha_{i_l+1}=\ldots=\alpha_{i_{l+1}} \not= \alpha_{i_{l+1}+1}$. 
Lemma \ref{lemma} then implies that $\pi^{-1}(N_{\beta_{i_l+1},\ldots,\beta_{i_{l+1}}})$ is 
a connected component of 
$(\pi')^{-1}(N_{\beta_{i_l'+1},\ldots,\beta_{j_l'}}))$ where $i_l'-j_l'=i_l-i_{l+1}$ and the 
$\beta_i$s satisfy the conditions 
$\beta_{i_l'+1}=\alpha_{i_l+1},\ldots,\beta_{j_l'}=\alpha_{i_{l+1}}$. Since 
$(\pi')^{-1}(N_{\beta_{i_l'+1},\ldots,\beta_{j_l'}})$ and 
$(\pi')^{-1}(N_{\beta_{i_{l+1}'+1},\ldots,\beta_{j_{l+1}'}})$ share a boundary component 
we must have $j_l'=i_{l+1}'$. This 
finishes the proof (we get $p=i_1'$).
\end{proof}


\section{Arithmetic manifolds with totally geodesic hypersurfaces}

The following construction is now standard and goes back at least to 
Millson's paper \cite{Millson}. We recall it here for the reader's 
convenience and because for an even $n$ it is not proved in 
\cite{GPS} that one can get noncommensurable manifolds this way.

\begin{prop}
For any $n\ge 2$ there exists countably many closed 
arithmetic hyperbolic $n$-manifolds which all contain some fixed 
$n-1$-hyperbolic manifold $S$ as an embedded nonseparating totally geodesic 
hypersurface and which are noncommensurable to each other.
\label{millson}
\end{prop}

\begin{proof}
The standard way to construct arithmetic hyperbolic manifolds which contain
totally geodesic hypersurfaces is as follows: let $F$ be a totally real number
field and $q$ a quadratic form in $n+1$ variables over $F$ such that $q$ is
definite positive at all real places of $F$ but one, where it has signature
$(1,n)$. Then the group of integer points $\Gamma=\SO(q,\mathcal{O}_F)$ is a
lattice in $\SO(1,n)$. If $q$ is written as $a_1x_1^2+\ldots+a_{n+1}x_{n+1}^2$
where $a_1,\ldots,a_n$ are totally positive and $a_{n+1}$ is negative at
exactly one real place,
then $\Gamma$ contains the subgroup $\Lambda$ associated to the quadratic 
form in $n$ variables $a_2x_2^2+\ldots+a_{n+1}x_{n+1}^2$. Define the principal 
congruence subgroup of level $\mathfrak{p}$, $\Gamma(\mathfrak{p})$, to 
be the kernel 
$\Gamma\cap\ker(\SL(n+1,\mathcal{O}_F)\rightarrow\SL(n+1,\mathcal{O}_F/\mathfrak{p}))$. It follows from the proof of 
Selberg's Lemma that for $|\mathfrak{p}|$ big enough the lattice 
$\Gamma(\mathfrak{p})$ is torsion-free, so that the subgroup 
$\Lambda(\mathfrak{p})$ gives rise to an imbedding 
$\Lambda(\mathfrak{p})\backslash\mathbb{H}^{n-1}\rightarrow\Gamma(\mathfrak{p})\backslash\mathbb{H}^{n}$ 
and a standard argument shows that this is in fact an embedding.
Moreover we can choose $\mathfrak{p}$ so that this hypersurface $S$ is 
non separating, that is $M-S$ is connected so that it is the interior 
of a compact manifold with two boundary components isometric to $S$. 
For all this see \cite[Section 2]{Millson} or \cite{GPS}. 
Finally, note that the isometry type of $S$ depends only on 
$a_2,\ldots,a_{n+1}$ and $\mathfrak{p}$.

The simplest exemple of the previous procedure 
is when $F=\mathbb{Q}$ and $a_1,\ldots,a_n>0,\, a_{n+1}<0$ but then the
manifolds obtained are noncompact for $n\ge 4$. However, if
$F=\mathbb{Q}(\sqrt{d})$ for a square-free rational integer $d>1$,
$a_1,\ldots,a_n\in\mathbb{Q}_+^*$ and $a_{n+1}/\sqrt{d}\in\mathbb{Q}^*$ then
$q$ is anisotropic over $F$ so that $\Gamma_q\backslash\mathbb{H}^3$ is
compact. In the sequel we make the simplest choice $d=2$, 
$a_1=\ldots=a_n=1$ and $a_{n+1}=-\sqrt{2}$. Put 
$q_a=ax_1^2+x_2^2+\ldots+x_N^2-\sqrt{2}x_{n+1}^2$. We want to show that there 
is an infinite number of $a\in\mathcal{O}_F$ such that the lattices 
$SO(q_a,\mathcal{O}_F)$ are all noncommensurable to each other. This 
boils down, by \cite[2.6]{GPS}, to show that the quadratic forms 
$q_a,\lambda q_{a'}$ are nonisometric for all $a,a'$ and 
$\lambda\in F^*$. Suppose first that $n$ is odd: then the discriminant of 
$\lambda q_a$ equals that of $q_a$ for all $\lambda$, thus we need only 
find a countable set $A\subset\mathcal{O}_F$ such that for all 
$a\not=a'\in A$ we have $a/a'\not\in (F^*)^2$. This 
is doable since the image of $\mathcal{O}_F-\{0\}$ in $F^*/(F^*)^2$ is 
easily seen to be infinite.

Now we treat the case of an even $n$: for that purpose we need to recall some 
definitions. 
Let $k$ be any field; for $u,v\in k^*$ the Hilbert symbol $(u,v)_k$ is defined 
in \cite[III,1.1]{CA} as $1$ if $1=uv^2+vy^2$ for some $x,y\in k$ and 
$-1$ otherwise. Then it is shown in \cite[IV, Th\'eor\`eme 2]{CA} that 
\begin{equation*}
\epsilon_k(q)=\prod_{i<j}(a_i,a_j)_k
\end{equation*}
is an isometry invariant of $q$ over $k$. Now we suppose that 
$k=\mathbb{Q}_p$ for a prime $p>2$, then for $a,b\in\mathbb{Z}_p$ we have 
$(a,b)_{\mathbb{Q}_p}=-1$ if and only if either $a$ (resp. $b$) has odd 
$p$-valuation and $b$ (resp. $a$) has even valuation but is not a square in 
$\mathbb{Q}_p$, or if $a,b$ have the same 
$p$-valuation mod 2 and $-a^{-1}b$ is a square in $\mathbb{Q}_p$ (see 
\cite[III, Th\'eor\`eme 2]{CA}). 

Now suppose that we have a rational prime $p$ such that 
$p= -1 \pmod{8}$. Then 2 is a biquadratic residue modulo $p$ according 
to Gauss' biquadratic reciprocity law, and since $p\not= 1\pmod{4}$ we know 
that $-1$ is not a square modulo $p$. By Hensel's lemma it follows that all 
this is still true in $\mathbb{Q}_p$. For any $\lambda\in\mathbb{Q}_p^*$ it 
follows that 
$(\lambda,\lambda)_{\mathbb{Q}_p}=1$, $(\lambda,-\sqrt{2}\lambda)_{\mathbb{Q}_p}=-1$ 
and more generally, for any $u\in\mathbb{Z}_p^*$, 
$(\lambda u,-\sqrt{2}\lambda)_{\mathbb{Q}_p}=-(\lambda u,\lambda)_{\mathbb{Q}_p}$
We can now compute:
\begin{eqnarray}
\begin{split}
\epsilon_{\mathbb{Q}_p}(\lambda q_u) &= (\lambda u,\lambda)_{\mathbb{Q}_p}^{n-1}\times(\lambda,-\sqrt{2}\lambda)_{\mathbb{Q}_p}^{n-1}\times(\lambda u,\-\lambda\sqrt{2})_{\mathbb{Q}_p}\\
                      &= -(\lambda u,\lambda)_{\mathbb{Q}_p}\times(\lambda,-\sqrt{2}\lambda)_{\mathbb{Q}_p}=1.
\label{hilbert}
\end{split}
\end{eqnarray}
Now if we let $p_k,k\ge 1$ be an infinite increasing sequence of primes all 
equal to -1 modulo 8 (such a sequence exists by Dirichlet's theorem) 
we get by \eqref{hilbert} that for all $l<k$ 
we have $\epsilon_{\mathbb{Q}_{p_k}}(q_{p_l})=1$. On the other hand, we have 
$\epsilon_{\mathbb{Q}_{p_k}}(q_{p_k})=(p_k,-\sqrt{2})_{\mathbb{Q}_{p_k}}=-1$. It 
follows that for $l<k$ and $\lambda\in F\subset\mathbb{Q}_{p_k}$ the 
quadratic forms $q_{p_k}$ and $\lambda q_{p_l}$ are not isometric over 
$F$, so that the lattices $\SO(q_{p_k},\mathcal{O}_F)$ are noncommensurable 
one to the other.

\end{proof}


\bibliography{bibli}

\bibliographystyle{plain}

\end{document}